\documentclass[oneside,english,11pt]{amsart}
\usepackage[T1]{fontenc}
\usepackage[latin9]{inputenc}
\usepackage{geometry}
\usepackage{color}
\usepackage{babel}
\usepackage{verbatim}
\usepackage{amsbsy}
\usepackage{amstext}
\usepackage{amsthm}
\usepackage{amssymb}
\usepackage[unicode=true]
 {hyperref}
\usepackage{lmodern}
\usepackage{microtype}

\makeatletter
\numberwithin{equation}{section}
\numberwithin{figure}{section}
\theoremstyle{plain}
\newtheorem{thm}{\protect\theoremname}[section]

\theoremstyle{definition}
\newtheorem{defn}[thm]{\protect\definitionname}
\theoremstyle{plain}
\newtheorem{lem}[thm]{\protect\lemmaname}
\newtheorem{prop}[thm]{\protect\propositionname}
\newtheorem{cor}[thm]{\protect\corollaryname}

\theoremstyle{definition}

\usepackage{dsfont}
\usepackage{mathtools}

\makeatother

\providecommand{\conjecturename}{Conjecture}
\providecommand{\definitionname}{Definition}
\providecommand{\lemmaname}{Lemma}
\providecommand{\problemname}{Problem}
\providecommand{\theoremname}{Theorem}
\providecommand{\corollaryname}{Corollary}
\providecommand{\remarkname}{Remark}
\providecommand{\propositionname}{Proposition}

\begin{document}
\title[Gap statistics and higher correlations for geometric progressions]{Gap statistics and higher correlations for geometric progressions modulo one}
\date{\today}
\author{Christoph Aistleitner, Simon Baker, Niclas Technau, and Nadav Yesha}
\begin{abstract}
Koksma's equidistribution theorem from 1935 states that for
Lebesgue almost every $\alpha>1$, the fractional parts of the geometric
progression $(\alpha^{n})_{n\geq1}$ are equidistributed modulo one. In the present paper we sharpen this result by
showing that for almost every $\alpha>1$, the correlations of all finite orders and hence the normalized gaps of $(\alpha^{n})_{n\geq1}$ mod 1 have a Poissonian limit distribution, thereby resolving a conjecture of the two first named authors. While an earlier approach used probabilistic methods in the form of martingale approximation, our reasoning in the present paper
is of an analytic nature and based upon the estimation of oscillatory integrals. This method is robust enough to allow us to extend our
results to a natural class of sub-lacunary sequences.
\end{abstract}

\thanks{CA is supported by the Austrian Science Fund (FWF), projects F-5512,
I-3466, I-4945 and Y-901. NT received funding 
from the European Research Council (ERC) 
under the European Union\textquoteright s Horizon 2020 research
and innovation program (Grant agreement No. 786758), and Austrian Science Fund (FWF) from project J 4464-N. NY is supported by the ISRAEL SCIENCE FOUNDATION (grant No. 1881/20).}
\address{Christoph Aistleitner: Institute of Analysis and Number Theory, TU Graz\\
Steyrergasse 30, 8010 Graz\\
Austria}
\email{aistleitner@math.tugraz.at}
\address{Simon Baker: School of Mathematics, University of Birmingham\\
Birmingham,  B15 2TT\\
UK}
\email{simonbaker412@gmail.com}
\address{Niclas Technau: School of Mathematical Sciences, Tel Aviv University\\
Tel Aviv 69978\\
Israel; \\
Department of Mathematics\\
University of Wisconsin--Madison\\
480 Lincoln Dr, Madison\\
WI-53706\\
USA}
\email{niclast@mail.tau.ac.il; technau@wisc.edu}
\address{Nadav Yesha: Department of Mathematics, University of Haifa\\
Haifa 3498838\\
Israel}
\email{nyesha@univ.haifa.ac.il}

\maketitle
\global\long\def\J{\mathcal{J}}%

\global\long\def\C{\mathcal{C}}%

\global\long\def\Lon{L_{1}}%

\global\long\def\Ltw{L_{2}}%

\global\long\def\Lth{L_{3}}%

\global\long\def\Lfo{L_{4}}%

\global\long\def\Lgi{L_{i}}%

\global\long\def\yon{y_{1}}%

\global\long\def\ytw{y_{2}}%

\global\long\def\yth{y_{3}}%

\global\long\def\yfo{y_{4}}%

\global\long\def\ygi{y_{i}}%

\global\long\def\Van{\mathrm{Van}}%

\global\long\def\A{A}%

\global\long\def\calS{\mathcal{S}}%

\global\long\def\la{\lambda}%

\global\long\def\y{y}%

\global\long\def\bfx{\mathbf{x}}%

\global\long\def\bfy{\mathbf{y}}%

\global\long\def\bfn{\mathbf{n}}%

\global\long\def\bfm{\mathbf{m}}%

\global\long\def\bfw{\mathbf{w}}%

\global\long\def\bfu{\mathbf{u}}%

\global\long\def\bfv{\mathbf{v}}%

\global\long\def\bfh{\mathbf{h}}%

\global\long\def\bfz{\mathbf{z}}%

\global\long\def\bfugi{u_{i}}%

\global\long\def\bft{\mathbf{t}}%

\global\long\def\bftau{\boldsymbol{\tau}}%

\global\long\def\bfPhi{\boldsymbol{\Phi}}%

\global\long\def\eqdef{\overset{\mathrm{def}}{=}}%

\global\long\def\bE{\mathbb{E}}%

\global\long\def\O{O}%

\global\long\def\E{\mathcal{E}}%

\global\long\def\V{\mathcal{V}}%

\global\long\def\L{\mathcal{L}}%

\global\long\def\I{\mathcal{I}}%

\global\long\def\calBN{\mathcal{B}_{k}}%

\global\long\def\calBNj{\mathcal{B}_{k,j}}%

\global\long\def\calBNp{\mathcal{N}_{k-1}^{\varepsilon}}%

\global\long\def\J{\mathcal{J}}%

\global\long\def\C{\mathcal{C}}%

\global\long\def\Lon{L_{1}}%

\global\long\def\Ltw{L_{2}}%

\global\long\def\Lth{L_{3}}%

\global\long\def\Lfo{L_{4}}%

\global\long\def\Lgi{L_{i}}%

\global\long\def\yon{y_{1}}%

\global\long\def\ytw{y_{2}}%

\global\long\def\yth{y_{3}}%

\global\long\def\yfo{y_{4}}%

\global\long\def\ygi{y_{i}}%

\global\long\def\Van{\mathrm{Van}}%

\global\long\def\A{A}%

\global\long\def\calS{\mathcal{S}}%

\global\long\def\la{\lambda}%

\global\long\def\y{y}%

\global\long\def\bfx{\mathbf{x}}%

\global\long\def\bfy{\mathbf{y}}%

\global\long\def\bfn{\mathbf{n}}%

\global\long\def\bfm{\mathbf{m}}%

\global\long\def\bfw{\mathbf{w}}%

\global\long\def\bfu{\mathbf{u}}%

\global\long\def\bfz{\mathbf{z}}%

\global\long\def\bfugi{u_{i}}%

\global\long\def\bft{\mathbf{t}}%

\global\long\def\bftau{\boldsymbol{\tau}}%

\global\long\def\bfPhi{\boldsymbol{\Phi}}%

\global\long\def\eqdef{\overset{\mathrm{def}}{=}}%

\global\long\def\bE{\mathbb{E}}%

\global\long\def\O{O}%

\global\long\def\E{\mathcal{E}}%

\global\long\def\V{\mathcal{V}}%

\global\long\def\L{\mathcal{L}}%

\global\long\def\I{\mathcal{I}}%

\global\long\def\calBN{\mathcal{B}_{k}}%

\global\long\def\calBNj{\mathcal{B}_{k,j}}%

\global\long\def\calBNp{\mathcal{N}_{k-1}^{\varepsilon}}%

\section{Introduction}

A sequence $(\vartheta_{n})_{n \geq1}\subseteq\left[0,1\right)$
is called uniformly distributed (or equidistributed) if each test interval $I\subseteq\left[0,1\right)$
contains asymptotically its ``fair share'' of points, that is, $(\vartheta_{n})_{n \geq1}$ is equidistributed when
\[
\frac{\#\left\{ n\leq N:\vartheta_{n}\in I\right\} }{N}\underset{N\rightarrow\infty}{\longrightarrow}\lambda\left(I\right)
\]
for all intervals $I \subseteq\left[0,1\right)$, where $\lambda$ denotes the Lebesgue measure. A sequence $(\vartheta_{n})_{n \geq1}$ of numbers in $\mathbb{R}$ is called uniformly distributed modulo one if the sequence of fractional parts $(\{\vartheta_n\})_{n \geq1}$ is uniformly distributed in $[0,1)$. The
classical theory of uniform distribution modulo one dates back to the early twentieth century, when
Weyl \cite{Weyl} laid its foundations in his famous paper of 1916.

One of the basic results in the area is Koksma's equidistribution theorem \cite{Koksma}, which states
that for $\lambda$-almost every $\alpha>1$, the sequence corresponding to the geometric progression $(\alpha^{n})_{n \geq1}$ is uniformly
distributed modulo one.
Such sequences with a ``typical'' value of $\alpha$ have been famously proposed by Knuth in his monograph \emph{The art of computer programming} \cite{Knuth} as examples of sequences showing strong pseudorandomness properties. Koksma's equidistribution theorem has been extended to so-called complete uniform distribution by Niederreiter and Tichy \cite{NT}, and quantitative equidistribution estimates were obtained in \cite{A}.  A version of Koksma's equidistribution theorem for self-similar measures was proved in \cite{Ba}. Describing the behaviour of $(\alpha^n)_{n\geq 1}$ for specific values of $\alpha$ is a challenging problem. 
A well known and open problem due to Mahler 
asks for the range of $(\{\xi (3/2)^n \})_{n \geq1}$,
where $\xi>0$ is a real parameter.
For more on this topic, and the study of the sequence $(\alpha^n)_{n\geq 1}$ modulo one, we refer the reader to \cite{B,D1,D2,FLP} and the references therein.

While the classical notion of equidistribution modulo one addresses the ``large-scale'' behaviour of the fractional parts of a sequence (counting the number of points in \emph{fixed} intervals), the study of the \emph{fine-scale} statistics
of sequences modulo one, i.e. statistics on the scale of the mean gap $ 1/N $, has attracted growing attention in recent years. Among the most popular fine-scale statistics are the $k$-point correlations and the nearest-neighbour gap distribution, which are defined as follows.

Let $\vartheta=(\vartheta_{n})_{n\geq1} \subseteq\mathbb{R}$ be a sequence, and let $k \geq 2$ be an integer.
Let $\calBN = \calBN(N)$ denote the set of integer $k$-tuples $(x_1, \dots, x_k)$ such that all components are in the range $\{1, \dots, N\}$ and no two components are equal.
For a compactly supported function $ f:\mathbb{R}^{k-1}\to\mathbb{R} $, the $k$-point correlation sum $R_{k}\left(f,\vartheta,N\right)$ is defined to be
\begin{equation} \label{rk_def}
R_{k}\left(f,\vartheta,N\right)\eqdef\frac{1}{N}\sum_{\bfx\in\calBN}\sum_{\bfm\in\mathbb{Z}^{k-1}}f\left(N\left(\Delta\left(\bfx,\vartheta\right)-\bfm\right) \right)
\end{equation}
where $\Delta\left(\bfx,\vartheta\right)$ denotes the difference vector
\begin{equation} \label{delta_def}
\Delta\left(\bfx,\vartheta\right)=\left(\vartheta_{x_{1}}-\vartheta_{x_{2}},\vartheta_{x_{2}}-\vartheta_{x_{3}},\ldots,\vartheta_{x_{k-1}}-\vartheta_{x_{k}}\right)\in\mathbb{R}^{k-1}.
\end{equation}
Let $ C_c^\infty(\mathbb{R}^{k-1}) $ denote the space of real-valued, smooth, compactly supported functions on $ \mathbb{R}^{k-1} $. If \[ \lim\limits_{N\to \infty}R_{k}\left(f,\vartheta,N\right) = \int_{\mathbb{R}^{k-1}}f(\bfx)~\textup{d}\bfx \]for all $ f\in C_c^\infty(\mathbb{R}^{k-1}) $  (equivalently, if $R_{k}\left(1_{\Pi},\vartheta,N\right) \to \textup{vol}(\Pi)$ as $N \to \infty$ for all axis-parallel boxes $\Pi$, where $ 1_\Pi $ is the indicator function of $ \Pi $), then we say that the $k$-point correlation of $ (\{\vartheta_{n}\})_{n\geq1} $ is ``Poissonian''. This notion alludes to the fact that such behaviour is in accordance with the (almost sure) behaviour of a Poisson process with intensity one.

To define the distribution of the so-called level spacings or nearest-neighbour gaps, i.e. gaps between consecutive elements of $ (\{\vartheta_{n}\})_{n\geq1} $, we need to consider the \emph{reordered} elements
$$
\vartheta_{(1)}^{N} \leq \vartheta_{(2)}^N \leq \dots \leq \vartheta_{(N)}^N \leq \vartheta_{(N+1)}^N,
$$
which we obtain as a reordering of $\{\vartheta_1\}, \dots, \{\vartheta_{N+1}\}$. Assume that the limit $N \to \infty$ of the function
$$
G(s,\vartheta,N) = \frac{1}{N} \# \left\{ n \leq N:~N \left( \vartheta_{(n+1)}^N - \vartheta_{(n)}^N \right)  \leq s \right\}
$$
exists for all $s \geq 0$. Then the limit function $G(s)$ is called the asymptotic distribution function of the level spacings (or, alternatively, of the nearest-neighbour gaps) of $ (\{\vartheta_{n}\})_{n\geq1} $. We say that the level spacings are {\em Poissonian} when $G(s) = 1 - e^{-s}$, which is in agreement with the well-known fact that the waiting times in the Poisson process are exponentially distributed.

The $k$-point correlation of order $k= 2, 3, \dots $, is also called the pair correlation, triple correlation, etc.\ Poissonian behaviour of these local statistics can be seen as a pseudorandomness property, since a sequence $X_1, X_2, \dots$ of independent, identically distributed random variables with uniform distribution on $[0,1)$ will almost surely have Poissonian correlations/gap distributions. Note that equidistribution is also traditionally seen as a pseudorandomness property, albeit on a ``global'' rather than on a ``local'' level.

Recently, the first two authors of the present paper proved that $(\{\alpha^n\})_{n\ge1}$ has Poissonian pair correlation for almost all $\alpha > 1$; see \cite{AB}. This is a refinement of Koksma's equidistribution theorem mentioned earlier, since it is known that a sequence with Poissonian pair correlations is necessarily equidistributed \cite{ALP,GL,mark}. In \cite{AB} it was conjectured that for almost all $ \alpha>1$, the $k$-point correlation of $(\{\alpha^n\})_{n\ge1}$ should also be Poissonian for all $k\ge2$, and that as a consequence the level spacings of $(\{\alpha^n\})_{n\ge1}$ are Poissonian as well. The main purpose of the present paper is to prove this conjecture.

\begin{thm}\label{thm: AB conjec is true}
For almost every $\alpha > 1$, the $k$-point correlation of $(\{\alpha^n\})_{n\ge1}$ is Poissonian for all $k \geq 2$.
\end{thm}

It is known that if the $k$-point correlation is Poissonian for all $k \geq 2$, then the level spacings are also Poissonian (see Appendix A of \cite{KR}). Thus as a direct consequence of Theorem \ref{thm: AB conjec is true} we obtain that for almost all $\alpha > 1$, the level spacings of $(\{\alpha^n\})_{n\ge1}$ are Poissonian. The same principle applies to other ordered statistics, such as the second-to-nearest neighbour gaps etc.; whenever all $k$-point correlations are Poissonian, then these ordered statistics also behave in accordance with the Poissonian model.

We deduce Theorem \ref{thm: AB conjec is true} from a more general result which, due to the robustness of our method, comes at essentially no extra cost. This more general result is the following.
\begin{thm}
\label{thm: main}Let $\left(a_{n}\right)_{n\geq1}$ be an increasing
sequence of positive real numbers such that
\begin{equation}
\lim_{n\rightarrow\infty}\frac{a_{n}}{\log n}=\infty, \label{eq: growth assumptions1}
\end{equation}
and such that
\begin{equation}
a_{n+1}-a_{n} \geq n^{-C}\label{eq: growth assumptions2}
\end{equation}
for some $C>0$ for all sufficiently large $n$. Then for almost every $\alpha>0$, the $k$-point correlation of the sequence
$ (\{e^{\alpha a_{n}}\})_{n\geq1} $
is Poissonian for all $k \geq 2$.
\end{thm}

Theorem \ref{thm: AB conjec is true} follows upon letting $ a_n = n $ in Theorem \ref{thm: main}, and observing that the map $ \alpha \mapsto e^\alpha $ from $ (0,\infty) $ to $ (1,\infty) $ preserves measure zero sets.
We remark that in Theorem \ref{thm: main}, the sequence $(e^{\alpha a_{n}})_{n\geq1}$ can very well be a sequence of
sub-exponential growth, that is
\[
\lim_{n\rightarrow\infty}\frac{e^{\alpha a_{n+1}}}{e^{\alpha a_{n}}}=1,
\]
and still have a gap statistic following the Poissonian model for almost all $\alpha>0$, e.g. by taking $ a_n = \sqrt{n} $, $ a_n=(\log (n+1))^2 $, $a_n= 
(\log \log (n + 2)) \log (n+1)$ etc.

To put our results into perspective, we mention some earlier related results. Fundamental work on the correlations of sequences in the unit interval was carried out by Rudnick, Sarnak and Zaharescu; see for example \cite{RS,RSZ,RZ}. As a general principle, proving Poissonian behaviour of the $k$-point correlation of a sequence becomes increasingly difficult when $k$ becomes large. For example, it is known \cite{RS} that $(\{n^2 \alpha\})_{n\ge1}$ has Poissonian pair correlation for almost all $\alpha$; the same is conjectured to be true for the triple correlation (and probably all higher correlations), but only partial results exist in this direction \cite{TW}. Rudnick and Zaharescu proved \cite{RZ2} that for a lacunary sequence of integers $(a_n)_{n\ge1}$, i.e.\ a sequence satisfying $\liminf\limits_{n\to \infty}\frac{a_{n+1}}{a_n} > 1$, for almost all $\alpha$, the $k$-point correlation of $(\{a_n \alpha\})_{n\ge1}$ is Poissonian for all $k \geq 2$. Recently, the last two authors of the present paper proved \cite{YT} that for every $k \geq 2$, the $k$-point correlation of $(\{n^\alpha\})_{n\ge1}$ is Poissonian for almost all $\alpha > 4k^2 -4k -1$; in the notation of Theorem \ref{thm: main} this corresponds to $ a_n = \log n $, so that assumption \eqref{eq: growth assumptions1} fails to hold. Results of a non-metric nature are particularly sparse. A marked exception is the sequence $(\{ \sqrt{n} \})_{n\ge1}$, for which the level spacings distribution is \emph{not} Poissonian, as was shown by Elkies and McMullen using methods from ergodic theory \cite {EM}; somewhat surprisingly, the pair correlation is, in fact, Poissonian \cite{EMV}. It is conjectured that the sequence $(\{n^2 \alpha\})_{n\ge1}$ has Poissonian pair correlation for every $\alpha$ which cannot be too well approximated by rational numbers, but again only partial results exist \cite{HB,TD}.

\section{Outline of the argument}

For the remaining part of this manuscript, we will only be dealing with the sequence $\vartheta\left(\alpha\right)=\left(e^{\alpha a_{n}}\right)_{n\geq1}$; we shall simply write $\Delta\left(\bfx,\vartheta\right)$ instead of $\Delta\left(\bfx,\vartheta(\alpha)\right)$, and $R_{k}\left(f,\alpha,N\right)$ instead of $R_{k}\left(f,\vartheta(\alpha),N\right)$.

The strategy to prove Theorem \ref{thm: main} is much in the spirit of \cite{YT} and will now be detailed. To begin
with, we restrict our attention to intervals of the special form
\begin{equation}
\J=\J\left(\A\right)\eqdef\left[A,A+1\right],\qquad\left(A>0\right)\label{def: J}
\end{equation}
which will remain fixed throughout the proof. It is certainly enough to demonstrate
that for each $k\ge2$ and $A>0$ the assertion of Theorem \ref{thm: main} holds,
for almost every $\alpha\in\J$. Note that $\J$, equipped with Borel sets and Lebesgue measure, forms a probability space, so it is natural to speak about expectations, variances, etc., of real-valued functions defined on $\J$.

To prove Theorem \ref{thm: main},
we show via
a variance estimate that $R_{k}$ concentrates around its mean value $\int_{-\infty}^{\infty}f\left(\bfx\right)\:\mathrm{d}\bfx$ with a reasonable error term; Theorem \ref{thm: main} will then follow from a routine argument (see \cite[Proposition 7.1]{YT}).
Using Poisson summation, we can phrase the variance estimate in
terms of oscillatory integrals of the form:
\begin{equation} \label{I_form}
I\left(\bfu,\bft\right)\eqdef\int_{\J}e\left(\phi\left(\bfu,\bft,\alpha\right)\right)\,\mathrm{d}\alpha
\end{equation}
where $ e(z)\eqdef e^{2\pi i z}$, and the \emph{phase function} $\phi$ is given by
\begin{equation}
	\label{eq:phi_def}
\phi\left(\bfu,\bft,\alpha\right)\eqdef\sum_{i \leq 2k} u_{i} e^{\alpha a_{t_{i}}}\qquad\bfu=\left(u_{1},\ldots,u_{2k}\right),\quad\bft=\left(t_{1},\ldots,t_{2k}\right).
\end{equation}
There are additional constraints on the integer vectors $\bfu, \bft$
which naturally arise from the analysis. More precisely, fixing $ \varepsilon>0 $, we will have $ \bfu = (\bfv,\bfw) $, $ \bft = (\bfx, \bfy) $ where $ \bfx, \bfy \in \calBN $, and $ \bfv,\bfw \in \mathcal{U}_{k}^{\varepsilon}$, where
\begin{equation}
\label{eq:Uke}
\mathcal{U}_{k}^{\varepsilon}
=\mathcal{U}_{k}^{\varepsilon}\left(N\right)
\eqdef \{ {\bf u}=\left(u_{1},\dots,u_{k}\right)
\in\mathbb{Z}^{k}:\,1\le\left\Vert 
\bfu\right\Vert _{\infty}\le2N^{1+\varepsilon},
\,u_{1}+\dots+u_{k}=0
\}.
\end{equation}
Our desired variance estimate
can, after a simple computation, be phrased as a bound for an average
of these integrals. It will then be shown
that
\begin{equation}
	\label{eq:Vk}
	V_{k}\left(N,\J,\varepsilon\right)\eqdef\frac{1}{N^{2k}}\sum_{\bfz=\left({\bf u},\bft\right)\in\left(\mathcal{U}_{k}^{\varepsilon}\right)^{2}\times\mathcal{B}_{k}^{2}}\left|I\left({\bf u},{\bf t}\right)\right|=O(N^{-1+\varepsilon}).
\end{equation}
To prove such a bound, we use a variant of Van der Corput's lemma which
requires us to guarantee that at each point $\alpha\in\J$ at least
some derivative of $\phi$ with respect to $\alpha$ is large. Such a ``repulsion'' property is captured by the function
\begin{equation} \label{van_def}
\Van_{\ell}\phi\left(\alpha\right)\eqdef\max_{i \leq \ell}
|\phi^{(i)}(\alpha)|,
\end{equation}
and we shall derive an acceptable lower bound on $\Van_{\ell}\phi$,
uniformly throughout $\J$.

The aforementioned repulsion principle (see Lemma \ref{lem: repulsion principle} below) is the driving force behind the argument, and the only part of the proof where assumptions \eqref{eq: growth assumptions1} and \eqref{eq: growth assumptions2} are used. Moreover, this way of reasoning is  robust, and the arithmetic that we require
is quite simple and essentially just the structure of the
real numbers plus quantitative growth and spacing conditions.
A technical complication which often arises in the study of $ k $-point correlation sums (see e.g. \cite{YT,RZ2}) is that we have to deal with ``degenerate'' configurations where $ \bfu $ and $ \bft $ are such that some of the terms in the function $ \phi\left(\bfu,\bft,\alpha\right) $ vanish; this will be handled by a combinatorial argument (see Proposition \ref{prop:comb_prop} below).

\section{Preliminaries}

In this section we collect the tools that we will use later and introduce
further notation.

\subsection{Notation}

Throughout the rest of this manuscript the implied constants may depend on the sequence $(a_n)_{n \geq 1}$ from the statement of Theorem \ref{thm: main}, as well as on $k,f,\J,\varepsilon, \eta$
and we shall not indicate this dependence explicitly. The dependence
on any other parameter will be indicated. The Bachmann--Landau $\O$  symbol, or interchangeably the Vinogradov symbols $\ll$ and $\gg$, have their usual meaning. Throughout the manuscript, $k$ is a fixed integer satisfying $k \geq 2$.

\subsection{Oscillatory integrals}

The bulk of our work is concerned with understanding the magnitude
of the one-dimensional oscillatory integrals
\[
I\left(\phi,\J\right)\eqdef\int_{\J}e\left(\phi\left(\alpha\right)\right)\,\mathrm{d\alpha}
\]
where $\phi:\J\rightarrow\mathbb{R}$ is a $C^{\infty}$-function, the
so called \emph{phase function}. The phase functions we are required to
understand are of the shape $ \phi(\alpha)=\phi\left(\bfu,\bft,\alpha\right) $ as in \eqref{eq:phi_def}.
We need the following variant of Van der Corput's lemma:
\begin{lem}
\label{lem: modified Van der Corput}Let $\phi:\J\rightarrow\mathbb{R}$
be a $C^{\infty}$-function. Fix $\ell\geq1$, and suppose that $\phi^{\left(\ell\right)}\left(\alpha\right)$
has at most $C$ zeros, and that the inequality
$\mathrm{Van}_{\ell}\phi\left(\alpha\right)\ge\lambda>0$
holds throughout the interval $\J$. Then the bound
\[
I\left(\phi,\J\right)\ll_{\ell,C}\lambda^{-1/\ell}
\]
holds when $\ell\ge2$, or when $\ell=1$ and $\phi'$ is monotone on $\J$.
\end{lem}
\begin{proof}
This can be found in \cite[Lemma 3.3]{YT}.
\end{proof}
Lemma \ref{lem: modified Van der Corput} requires a bound on the number of zeros for the derivatives of $\phi$. For this we prove the following  which is a very minor modification of \cite[Lemma 4.3]{YT}.
\begin{lem}\label{lem: bound number of zeros of phase function}
Let $\psi(\alpha)=\sum_{i\leq \ell}u_ie^{\alpha x_i}$ for $\bfu=(u_1,\ldots,u_\ell)\in \mathbb{R}^{\ell}_{\neq 0}$ and $\bfx=(x_1,\ldots,x_\ell)\in \mathbb{R}^\ell$ such that $x_1<\cdots <x_{\ell}$. Then $\psi$ has at most $\ell -1 $ zeros in $ \mathbb{R} $.
\end{lem}
\begin{proof}
	
	We argue by induction on $\ell$.
	For $\ell=1$ the correctness of the
	statement is clear.
	Assume that the lemma is true for $\ell-1$ ($\ell\ge2)$,
	and let
	\[
	\psi\left(\alpha\right)=\sum_{i\leq \ell}u_{i}e^{\alpha x_i}.
	\]
	The zeros of $\psi$ are exactly the zeros of the function
	\[
	\tilde{\psi}\left(\alpha\right)=
	\sum_{i\leq \ell-1}\tilde{u}_{i} e^{\alpha \tilde{x}_i}+1,
	\]
	where $\tilde{u}_{i}=\frac{u_{i}}{u_{\ell}}$, and
	$\tilde{x}_{i}= x_{i}- x_{\ell}$
	($1\le i\le \ell-1$), since $\psi\left(\alpha\right)
	= u_{\ell} e^{\alpha x_{\ell}} \tilde{\psi}\left(\alpha\right)$.
	Moreover,
	\[
	\tilde{\psi}'\left(\alpha\right)=
	\sum_{i\leq \ell-1}v_{i} e^{\alpha \tilde{x}_i},
	\]
	where $v_{i}=\tilde{u}_{i} \tilde{x}_{i}$ ($1\le i\le \ell-1$).
	
	Clearly, the numbers $v_{1},\dots,v_{\ell-1}$ are nonzero,
	and the $\tilde{x}_{1},\dots,\tilde{x}_{\ell-1}$
	are distinct. Therefore, by the induction hypothesis, $\tilde{\psi}'$
	has at most $\ell-2$ zeros. Hence, by Rolle's theorem, $\tilde{\psi}$
	has at most $\ell-1$ zeros, completing the proof.
\end{proof}

\section{The repulsion principle}
\begin{lem}
\label{lem: spectral norm bound} Let $\ell$ be a positive integer. Let $\gamma > 0$. Let $0 < x_{1} < x_{2} < \ldots < x_{\ell}$ be real numbers
such that $x_{i+1}-x_{i}\geq \gamma$ for $1 \leq i \leq \ell-1$. Then the matrix
\begin{equation}
M=M\left(x_{1},\ldots,x_{\ell}\right)=\begin{pmatrix}x_{1} & \ldots & x_{\ell}\\
\vdots & \ddots & \vdots\\
x_{1}^{\ell} & \ldots & x_{\ell}^{\ell}
\end{pmatrix}\label{def: matrix}
\end{equation}
is invertible and the operator norm  $\left\Vert \cdot\right\Vert _{\infty}$ of its inverse satisfies
\[
\left\Vert M^{-1}\right\Vert _{\infty}\ll_{\ell} x_{\ell}^{\ell-1}  x_1^{-1} \left(\frac{1}{\gamma} \right)^{\ell-1}.
\]
\end{lem}

\begin{proof}
The conclusion is trivial when $\ell=1$, so we will assume that $\ell \geq 2$. The matrix $M$ is the transpose of a scaled Vandermonde matrix;
the entry $m_{ij}$ of its inverse $M^{-1}$ is given by (see, e.g. \cite[Ex. 40]{Knuth})
$$
m_{ij} = (-1)^{j-1}  \frac{\sum\limits_{\substack{1 \leq m_1 < \dots < m_{\ell-j} \leq \ell,\\ m_1, \dots, m_{\ell-j} \neq i}} x_{m_1} \cdots x_{m_{\ell-j}}}{x_i \prod\limits_{\substack{1 \leq m \leq \ell,\\m \neq i}} (x_m - x_i)}.
$$
Hence
\begin{equation}\label{size_elem}
|m_{ij}| \ll_\ell x_\ell^{\ell-1}  x_1^{-1} \left( \frac{1}{\gamma} \right)^{\ell-1}
\end{equation}
for all $1 \leq i,j \leq \ell$.
It is well-known that the maximum norm $\left\Vert \cdot\right\Vert _{M}$ (maximal absolute value of a matrix entry)
dominates the operator norm $\left\Vert \cdot\right\Vert _{\infty}$, i.e. we have $\left\Vert \cdot\right\Vert _{\infty}\ll_\ell \left\Vert \cdot\right\Vert _{M}$.
Thus, \eqref{size_elem} gives
$$
\left\Vert M^{-1}\right\Vert _{\infty}\ll_\ell x_\ell^{\ell-1} x_1^{-1} \left(\frac{1}{\gamma} \right)^{\ell-1} ,
$$
as desired.
\end{proof}
As a consequence, we are now able to prove the enunciated repulsion
principle. In the statement of the following lemma, as throughout the proof, $(a_n)_{n\geq 1}$ is the sequence from the statement of Theorem \ref{thm: main}. Recall that by assumption $(a_n)_{n\geq 1}$ satisfies \eqref{eq: growth assumptions1} and \eqref{eq: growth assumptions2}, which will be used in the proof of the lemma. Recall also the definition of $\phi\left(\mathbf{u},\mathbf{t},\alpha\right)$ in \eqref{eq:phi_def} and the definition of $\Van_\ell$ in \eqref{van_def}.
\begin{lem}[Repulsion principle]
\label{lem: repulsion principle} Let $\ell$ be a positive integer such that $ \ell \le 2k $. 
 Let $\bfu\in\mathbb{Z}_{\neq0}^{\ell}$, and let $\bft = (t_1,\ldots,t_\ell)\in \mathbb{N}^\ell$ be such that $t_1 < \dots  < t_\ell$. Then for any (arbitrarily large) $\eta>0$,
\begin{equation} \label{poss}
\min_{\alpha\in\J}\Van_{\ell}\left(\phi\left(\mathbf{u},\mathbf{t},\alpha\right)\right)\gg t_\ell^{\eta}.
\end{equation}
\end{lem}
The implied constant in \eqref{poss} depends on $ \eta $, the sequence $(a_n)$, the interval $\J$ and the parameter $k$, which throughout the proof are assumed to be fixed.
\begin{proof}
Let $\alpha\in\J$.
To make the underlying structure more transparent, we denote $\bftau=(\partial_{\alpha}^{j}\phi\left(\mathbf{u},\mathbf{t},\alpha\right))_{j=1,\ldots,\ell}$,
$\bfw=(u_{i}e^{\alpha a_{t_{i}}})_{i=1,\ldots,\ell}$, and $M=M(a_{t_1},\dots,a_{t_\ell})$
 as in (\ref{def: matrix}). Then \[ \Van_{\ell}\left(\phi\left(\mathbf{u},\mathbf{t},\alpha\right)\right)=\| \bftau \|_\infty,  \]and
\begin{equation} \label{equ_rew}
 \bftau  = M\bfw .
\end{equation}
Note that we have $a_{t_{i+1}} - a_{t_i} \gg t_\ell^{-C}$ for some fixed positive constant $C$ by assumption \eqref{eq: growth assumptions2}. Thus by Lemma \ref{lem: spectral norm bound} and \eqref{equ_rew} we have
	\[
	\left\Vert \bfw\right\Vert _{\infty}\leq\left\Vert M^{-1}\right\Vert _{\infty}\left\Vert \bftau\right\Vert _{\infty}  \ll a_{t_{\ell}}^{\ell-1} t_\ell^{C(\ell-1)} \left\Vert \bftau\right\Vert _{\infty}
	\]
(we have used the bound $ a_{t_1}^{-1} \ll a_1^{-1} \ll 1$). Now note that
	$$
	\left\Vert \bfw\right\Vert _{\infty} \geq |u_\ell| e^{\alpha a_{t_\ell}} \geq e^{\alpha a_{t_\ell}}.
	$$ Combining the two equations above we obtain
	\begin{equation} \label{finite solutions}
	e^{\alpha a_{t_\ell}}\ll a_{t_{\ell}}^{\ell-1} t_\ell^{C(\ell-1)} \left\Vert \bftau\right\Vert _{\infty}.
	\end{equation}
Recall that the interval $\J$ and therefore $\alpha$ are bounded away from 0 by assumption, so that $a_{t_{\ell}}^{\ell-1} \ll e^{\frac{\alpha a_{t_\ell}}{2}} $, and hence \eqref{finite solutions} gives
\begin{equation} \label{finite solutions2}
	\left\Vert \bftau\right\Vert _{\infty} \gg e^{ \frac{\alpha a_{t_\ell}}{2}} t_\ell^{-C(\ell-1)}.
	\end{equation}
By assumption \eqref{eq: growth assumptions1} it follows that $e^{ \frac{\alpha a_{t_\ell} }{2}} \gg t_\ell^{\eta+C(\ell-1)}$. Inserting this into \eqref{finite solutions2} gives the required bound \eqref{poss}.
\end{proof}

As a corollary, we get the required bound for the integral $ I(\bfu,\bft) $ (recall the notation \eqref{I_form}).

\begin{cor}
\label{cor: IntegralBound} Let $\ell$ be a positive integer such that $ \ell \le 2k $.
Let $\bfu\in\mathbb{Z}_{\neq0}^{\ell}$, and let $\bft = (t_1,\ldots,t_\ell)\in \mathbb{N}^\ell$ be such that $t_1 < \dots  < t_\ell$. Then for any (arbitrarily large) $\eta>0$,
\begin{equation}
	I(\bfu, \bft) \ll t_\ell^{-\eta}.
\end{equation}
\end{cor}

\begin{proof}
	This is an immediate consequence of Lemma \ref{lem: modified Van der Corput}, Lemma \ref{lem: bound number of zeros of phase function} and Lemma \ref{lem: repulsion principle}.
\end{proof}

\section{Variance estimates}
We begin this section with our definition of the variance of the $ k $-point correlation sum $ R_{k}\left(f,\alpha,N\right) $ with respect to $ \alpha \in \J $. Recall that $\mathcal{B}_{k}=\mathcal{B}_{k}\left(N\right)$ is the
set of integer $k$-tuples $\left(x_{1},\dots,x_{k}\right)$ such
that $1\le x_{i}\le N$ for all $i=1,\dots,k$ and such that no two
components $x_{i}$ are equal.

\begin{defn}
	The variance of
	the $k$-point correlation sum $R_{k}\left(f,\alpha,N\right)$ with
	respect to the interval $\mathcal{\J}$ is defined as
	\[
	\mathrm{Var}\left(R_{k}\left(f,\cdot,N\right),\J\right)\eqdef\int_{\mathcal{\mathcal{\J}}}\left(R_{k}\left(f,\alpha,N\right)-C_{k}\left(N\right)\int_{\mathbb{R}^{k-1}}f\left({\bf x}\right)\,\text{d}{\bf x}\right)^{2}\,\mathrm{d}\alpha
	\]
	where
	\begin{equation}
		C_{k}\left(N\right)\eqdef\frac{\# \mathcal{B}_{k}}{N^k}=\left(1-\frac{1}{N}\right)\cdots\left(1-\frac{k-1}{N}\right).\label{eq:C_k(N)}
	\end{equation}
	The reason for the combinatorial factor (\ref{eq:C_k(N)}) will be
	apparent in the proof below.
\end{defn}

The goal of this section is to show that the variance $ \mathrm{Var}\left(R_{k}\left(f,\cdot,N\right),\J\right) $ decays polynomially in $ N $:

\begin{prop}
	\label{prop:MainProp}For all $ \varepsilon>0 $, we have \[ 	\mathrm{Var}\left(R_{k}\left(f,\cdot,N\right),\J\right)=O(N^{-1+\varepsilon}). \]
\end{prop}

The first routine step will be to express $ R_k $ in terms of an exponential sum. 
Fix $\varepsilon>0$, and set 
$$
\calBNp=\calBNp\left(N\right) = \left\{ \mathbf{n} \in \mathbb{Z}^{k-1}: ~1\le\|\mathbf{n}\|_\infty \leq N^{1 + \varepsilon} 
\right\}.
$$
For the statement of the following lemma, recall the definition of the difference vector $\Delta\left(\bfx,\alpha\right) = \Delta\left(\bfx,\vartheta(\alpha)\right)$ in \eqref{delta_def}.

\begin{lem}
	\label{lem:TruncatedPoisson}For all $\eta>0$ we have
	\begin{align}
		R_{k}\left(f,\alpha,N\right) & =C_{k}\left(N\right)\int_{\mathbb{R}^{k-1}}f\left({\bf x}\right)\,\mathrm{d}{\bf x}\label{eq: truncated Poisson form}\\
		& +\frac{1}{N^{k}}\sum_{\bfx\in\mathcal{B}_{k}}\sum_{\bfn\in\calBNp}\widehat{f}\left(\frac{\bfn}{N}\right)e\left(\left\langle \Delta\left(\bfx,\alpha\right),\bfn\right\rangle \right)+\O(N^{-\eta})\nonumber 
	\end{align}
	as $N\to\infty$.
\end{lem}

\begin{proof}
	By the Poisson summation formula,
	\[
	R_{k}\left(f,\alpha,N\right)=\frac{1}{N^{k}}\sum_{\bfx\in\mathcal{B}_{k}}\sum_{\bfn\in\mathbb{Z}^{k-1}}\widehat{f}\left(\frac{\bfn}{N}\right)e\left(\left\langle \Delta\left(\bfx,\alpha\right),\bfn\right\rangle \right).
	\]
	Formula (\ref{eq: truncated Poisson form}) now easily follows by separating the zero-th term and using the fact that the Fourier coefficients of any $f\in C_{c}^{\infty}(\mathbb{R}^{k-1})$ decay to zero faster than the reciprocal of any polynomial, see the proof of \cite[Lemma 3.4]{YT}.
\end{proof}
Given ${\bf n}\in\mathbb{Z}^{k-1}$, we define the vector $\bfh\left({\bf n}\right)=\left(h_{1}\left({\bf n}\right),\dots,h_{k}\left({\bf n}\right)\right)\in\mathbb{Z}^{k}$
by the rule
\[
h_{i}\left(\bfn\right)\eqdef\begin{cases}
	n_{1}, & \mathrm{if}\,i=1,\\
	n_{i}-n_{i-1}, & \mathrm{if}\,2\leq i\leq k-1\\
	-n_{k-1}, & \mathrm{if}\,i=k.
\end{cases},
\]
This definition is motivated by the identity
\begin{equation}
	\label{eq:partial_summation_identity}
\left\langle \Delta\left(\bfx,\alpha\right),\bfn\right\rangle =\phi\left({\bf h}\left(\bfn\right),\bfx,\alpha\right).
\end{equation}
Note that the linear map ${\bf n}\mapsto{\bf h}\left({\bf n}\right)$
is injective. Moreover, it satisfies the bound
\begin{equation}
	\left\Vert {\bf {\bf h}}\left({\bf n}\right)\right\Vert _{\infty}\le2\left\Vert {\bf n}\right\Vert _{\infty}\label{eq:u_n_range}
\end{equation}
and the relation
\begin{equation}
	\sum_{i=1}^{k}h_{i}\left(\bfn\right)=0.\label{eq:u_n_relation}
\end{equation}
Let 
\[
\mathcal{U}_{k}^{\varepsilon}=\mathcal{U}_{k}^{\varepsilon}\left(N\right)=\left\{ {\bf u}=\left(u_{1},\dots,u_{k}\right)\in\mathbb{Z}^{k}:\,1\le\left\Vert \bfu\right\Vert _{\infty}\le2N^{1+\varepsilon},\,u_{1}+\dots+u_{k}=0\right\} ,
\]
and note that the relations (\ref{eq:u_n_range}), (\ref{eq:u_n_relation})
imply that ${\bf h}\left({\bf n}\right)\in\mathcal{U}_{k}^{\varepsilon}$
whenever ${\bf{n}} \in\calBNp$.
\begin{lem}
	\label{lem: variance for higher corr via oscil inte} For all $\eta>0$,
	we have
	\begin{align}
		\mathrm{Var}\left(R_{k}\left(f,\cdot,N\right),\J\right) & \ll V_{k}\left(N,\J,\varepsilon\right)+N^{-\eta}\label{eq: rep of variance remainder}
	\end{align}
	as $N\to\infty$, where $ V_{k}\left(N,\J,\varepsilon\right) $ is given in \eqref{eq:Vk}.
\end{lem}

\begin{proof}
	By Lemma \ref{lem:TruncatedPoisson}, for all $\tilde{\eta}>0$ we
	have
	\[
	\mathrm{Var}\left(R_{k}\left(f,\cdot,N\right),\J\right)=\int_{\J}\Bigl(N^{-k}\sum_{\bfx\in\mathcal{B}_{k}} \sum_{\bfn\in\calBNp}
	\widehat{f}\left(\frac{\bfn}{N}\right)e\left(\left\langle \Delta\left(\bfx,\alpha\right),\bfn\right\rangle \right)+\O(N^{-\tilde{\eta}})\Bigr)^{2}\:\mathrm{d}\alpha.
	\]
	Expanding the square and taking $\tilde{\eta}$ sufficiently large, the bound  $\widehat{f}\ll1$ readily implies that
	for all $\eta>0$,
	\begin{align*}
		 \mathrm{Var}\left(R_{k}\left(f,\cdot,N\right),\J\right)=I+O\left(N^{-\eta}\right)		
	\end{align*}
	where
	\[
	I=N^{-2k}\sum_{\substack{\bfx,{\bf y}\in\mathcal{B}_{k},\\
			\bfn,{\bf m}\in\calBNp
		}
	}\widehat{f}\left(\frac{\bfn}{N}\right)\widehat{f}\left(\frac{{\bf m}}{N}\right)\int_{\J}e\left(\left\langle \Delta\left(\bfx,\alpha\right),\bfn\right\rangle +\left\langle \Delta\left({\bf y},\alpha\right),{\bf m}\right\rangle \right)\:\mathrm{d}\alpha.
	\]
	By the identity \eqref{eq:partial_summation_identity} and by the injectivity of the map ${\bf n}\mapsto{\bf h}\left({\bf n}\right)$
	we conclude that
	\[
	I\ll N^{-2k}\sum_{\substack{\bfx,{\bf y}\in\mathcal{B}_{k},\\
			\bfn,{\bf m}\in\calBNp
		}
	}\left|\int_{\J}e\left(\phi\left({\bf h}\left(\bfn\right),\bfx,\alpha\right)+\phi\left({\bf h}\left(\bfm\right),{\bf y},\alpha\right)\right)\:\mathrm{d}\alpha\right| \ll V_{k}\left(N,\J,\varepsilon\right)
	\]
	which gives the claimed bound.
\end{proof}

We will now bound $V_{k}\left(N,\J,\varepsilon\right)$ using a combinatorial
argument. Combined with Lemma \ref{lem: variance for higher corr via oscil inte}, this will give Proposition \ref{prop:MainProp}.
\begin{prop}
	\label{prop:comb_prop}
	We have~ 
	\[
	V_{k}\left(N,\J,\varepsilon\right)=O(N^{-1+\left(2k-1\right)\varepsilon}).
	\]
\end{prop}

\begin{proof}
Let
	\[
	\left[k\right]\eqdef\left\{ 1,\dots,k\right\} .
	\]
	Let $\I_{1},\I_{1}',\I_{2},\I_{2}',\I_{3},\I_{3}'\subseteq\left[k\right]$
	be (possibly empty) sets of indices. Fixing 
	\[
	\bftau\eqdef\left(\I_{1},\I_{1}',\I_{2},\I_{2}',\I_{3},\I_{3}'\right),
	\]
	we denote by $\V_{k}^{\varepsilon}\left(\bftau\right)$ the set of vectors
	$\left({\bf u},{\bf t}\right)=\left(\left({\bf v},{\bf {\bf w}}\right),\left({\bf x},{\bf y}\right)\right)\in\left(\mathcal{U}_{k}^{\varepsilon}\right)^{2}\times\mathcal{B}_{k}^{2}$
	for which 
	\begin{align*}
	\{i\in[k]:\,\exists_{j(i)\in[k]}:\,x_{i}=y_{j(i)}\} & =\mathcal{I}_{1},\\
	\{j\in[k]:\,\exists_{i(j)\in[k]}:\,x_{i\left(j\right)}=y_{j}\} & =\mathcal{I}_{1}',\\
	\{i\in\I_{1}:v_{i}+w_{j(i)}=0,\mathrm{where}\,j(i)\,\mathrm{is}\,\mathrm{s.t.}\,x_{i}=y_{j(i)}\} & =\I_{2},\\
	\{j\in\I_{1}':v_{i\left(j\right)}+w_{j}=0,\mathrm{where}\,i(j)\,\mathrm{is}\,\mathrm{s.t.}\,x_{i\left(j\right)}=y_{j}\} & =\I_{2}',\\
	\{i\in\left[k\right]\setminus\mathcal{I}_{1}:v_{i}=0\} & =\I_{3},\\
	\{j\in\left[k\right]\setminus\mathcal{I}_{1}':w_{j}=0\} & =\I_{3}'.
	\end{align*}
	If $\V_{k}^{\varepsilon}\left(\bftau\right)$ is non-empty, then clearly
	$\#\mathcal{I}_{1}=\#\mathcal{I}_{1}'$ and $\#\I_{2}=\#\I_{2}'$. Amongst the list of $2k$ variables $x_{1},\dots,x_{k},y_{1},\dots,y_{k},$
	exactly $2k-\#\mathcal{I}_{1}$ distinct variables appear (to see this, recall that by the definition of $\mathcal{B}_{k}$ all numbers $x_1, \dots, x_k$ are distinct, and similarly all numbers $y_1, \dots, y_k$ are distinct). As such if we group similar terms in the corresponding phase function we have
	\begin{align}
	\phi\left({\bf u},{\bf t},\alpha\right) & =v_{1}e^{\alpha x_{1}}+\dots+v_{k}e^{\alpha x_{k}}+w_{1}e^{\alpha y_{1}}+\dots+w_{k}e^{\alpha y_{k}}\label{eq:phi_group}\\
	& =\sum_{i\in\left[k\right]\setminus\left(\mathcal{I}_{1}\cup\mathcal{I}_{3}\right)}v_{i}e^{\alpha x_{i}}+\sum_{j\in\left[k\right]\setminus\left(\mathcal{I}_{1}'\cup\mathcal{I}_{3}'\right)}w_{j}e^{\alpha y_{j}}+\sum_{i\in\mathcal{I}_{1}\setminus\mathcal{I}_{2}}\left(v_{i}+w_{j(i)}\right)e^{\alpha x_{i}},\nonumber 
	\end{align}
	and the number of non-vanishing terms is $$l \eqdef 2k-\#\mathcal{I}_{1}-\#\I_{2}-\#\I_{3}-\#\I_{3}'.$$  Now let us consider the constraints on the variables $ v_1 \dots ,v_k, w_1,\dots,w_k $:
	\begin{itemize}
		\item The constraints $v_{i}=0$ $(i\in\I_{3})$ and $v_{1}+\dots+v_{k}=0$
		(recall that ${\bf v}\in\mathcal{U}_{k}^{\varepsilon}$) determine $\#\I_{3}+1$
		of the variables $v_{1},\dots,v_{k}$ in terms of the other variables;
		note that ${\bf v}\ne(0,\dots,0)$, so that $\#\I_{3}<k-1.$
		\item The constraints $w_{j}=0$ $(j\in\I_{3}'$) and $w_{j}=-v_{i\left(j\right)}$
		($j\in\I_{2}'$) determine $\#\I_{2}'+\#\I_{3}'$ of the variables
		$w_{1},\dots,w_{k}$ in terms of the variables $v_{i}$.
		\item Since $w \in \mathcal{U}_{k}^{\varepsilon}$, we also have the constraint $ w_1+\dots+w_k=0 $ which is either contained in the previous constraints (this happens if and only if $ l=0 $), or determines one more variable.
	\end{itemize}
	To conclude, the constraints on the variables $v_1,\dots,v_k,w_1,\dots,w_k$ determine at least 
	$$
	\left(\#\I_{3}+1\right)+\left(\#\I_{2}'+\#\I_{3}'\right)
	$$
	many of these variables. As such if we denote by $m$ the number of independent variables remaining then 
	\begin{equation}
	\label{m bound}
	m\leq 2k-\#\I_{3}-\#\I_{2}'-\#\I_{3}'-1.
	\end{equation}We relabel these independent variables by $u_{1},\dots,u_{m}.$ Suppose $u_{1},\dots,u_{m}$ are given, then we let $\bfu^{*}(u_1,\dots,u_m)$ denote the unique element of $\mathbb{Z}^{2k}$ for which the conditions corresponding to $\mathcal{I}_{2}'$, $\mathcal{I}_3$ and $\mathcal{I}_{3}'$ are satisfied, and the equations $v_1+\cdots + v_k=0$ and $w_1+\cdots + w_k=0$ are satisfied. We now proceed via a case analysis based upon the value of $l$ to obtain a uniform upper bound for $\sum_{\left({\bf u},{\bf t}\right)\in\V_{k}^{\varepsilon}\left(\bftau\right)}\left|I\left({\bf u},{\bf t}\right)\right|$.\\
	
	\noindent \textbf{Case 1, $l=0$.} If $ l=0 $ then $ \phi(\bfu, \bft, \alpha)=0 $ so that $ I(\bfu,\bft)=1 $. By the above considerations we may conclude that 
		\begin{align*}
	\sum_{\left({\bf u},{\bf t}\right)\in\V_{k}^{\varepsilon}\left(\bftau\right)}\left|I\left({\bf u},{\bf t}\right)\right|&\ll\sum_{\stackrel{|u_{1}|,\dots,|u_{m}|\le2N^{1+\varepsilon}}{\bfu^{*}(u_1,\dots,u_m)\in (\mathcal{U}_{k}^{\epsilon})^{2}}}\sum_{\stackrel{1\le t_{1},\dots,t_{2k-\#\mathcal{I}_{1}}\le N}{t_i\neq t_j}} 1\\
	&\ll  N^{m\left(1+\varepsilon\right)+2k-\#\mathcal{I}_{1}}\le N^{\left(2k-1\right)\left(1+\varepsilon\right)}.
\end{align*}In the final line we used \eqref{m bound} and the fact that $l=0$.\\

\noindent \textbf{Case 2, $l\geq 1$.} 	Assuming $l\geq 1$ we relabel the distinct variables $x_{i},y_{j}$ appearing on the r.h.s. of (\ref{eq:phi_group}) by $t_{1},t_{2},\dots,t_{l}.$ We also relabel by $ s_1,\dots,s_r $ the 
\begin{equation}
	\label{r equation}
	r\eqdef \#\I_{2}+\#\I_{3}+\#\I_{3}' 
	\end{equation}variables $ x_i,y_j$ from our list of distinct variables which do not appear on the r.h.s. of (\ref{eq:phi_group}) because their corresponding exponentials $ e^{\alpha s_i} $ have zero coefficients. Moreover we denote by $\mathbf{t}^{*}(t_1,\ldots,t_{l},s_1,\ldots,s_r)$ the unique element of $\mathcal{B}_{k}^{2}$ determined by $t_1,\ldots,t_{l},$ $s_1,\ldots,s_r,$ and the conditions imposed by $\bftau$. Note that by Corollary \ref{cor: IntegralBound}, we always have $$I\left(\bfu^{*}(u_1,\dots,u_m),{\mathbf{t}^{*}(t_1,\ldots,t_{l},s_1,\ldots,s_r)}\right)\ll|\max_{i} t_{i}|^{-\eta}$$ for any $\eta>0$. By the above considerations we may conclude that
	\begin{align*}
	&\sum_{\left({\bf u},{\bf t}\right)\in\V_{k}^{\varepsilon}\left(\bftau\right)}\left|I\left({\bf u},{\bf t}\right)\right|\\
	&\ll\sum_{\stackrel{|u_{1}|,\dots,|u_{m}|\le2N^{1+\varepsilon}}{\bfu^{*}(u_1,\dots,u_m)\in (\mathcal{U}_{k}^{\epsilon})^{2}}}\sum_{\stackrel{1\le s_{1},\dots,s_{r}\le N}{s_i\neq s_j}}\sum\limits_{\stackrel{1\le t_{1},\dots,t_{l}\le N}{t_{i}\neq t_{j}, t_{i}\neq s_{j}}} |I\left(\bfu^{*}(u_1,\dots,u_m),{\mathbf{t}^{*}(t_1,\ldots,t_{l},s_1,\ldots,s_r)}\right)|\\
	&\ll\sum_{\stackrel{|u_{1}|,\dots,|u_{m}|\le2N^{1+\varepsilon}}{\bfu^{*}(u_1,\dots,u_m)\in (\mathcal{U}_{k}^{\epsilon})^{2}}}\sum_{1\le s_{1},\dots,s_{r}\le N}\sum\limits_{1\le t_{1},\dots,t_{l}\le N}|\max_{i} t_{i}|^{-\eta}\\
	&\ll\sum_{|u_{1}|,\dots,|u_{m}|\le2N^{1+\varepsilon}}\sum_{1\le s_{1},\dots,s_{r}\le N}1
	\ll N^{m\left(1+\varepsilon\right)+r}\le N^{\left(2k-1\right)\left(1+\varepsilon\right)}.
	\end{align*}In the last line we used \eqref{m bound} and \eqref{r equation}.\\
	
Combining the above cases, we have shown that for any value of $l$ we always have $$\sum_{\left({\bf u},{\bf t}\right)\in\V_{k}^{\varepsilon}\left(\bftau\right)}\left|I\left({\bf u},{\bf t}\right)\right|\ll N^{\left(2k-1\right)\left(1+\varepsilon\right)}.$$ Therefore summing over all $O\left(1\right)$ configurations $\bftau$,
we conclude that 
\[
V_{k}\left(N,\J,\varepsilon\right)=\frac{1}{N^{2k}}\sum_{\tau}\sum_{\left({\bf u},{\bf t}\right)\in\V_{k}^{\varepsilon}\left(\bftau\right)}\left|I\left({\bf u},{\bf t}\right)\right|\ll N^{-1+\left(2k-1\right)\varepsilon}.
\] This completes our proof. 
\end{proof}

\section{Proof of Theorem \ref{thm: main}}

Theorem \ref{thm: main} can be deduced from Proposition \ref{prop:MainProp} following a standard argument whose proof in a fairly general setting was given in \cite{YT}.

\begin{prop}[{\cite[Proposition 7.1]{YT}}]
	\label{prop:L2->a.s.}
	Fix $ k\ge2 $, $ J\subset \mathbb{R} $ a bounded interval, and a sequence $ c_k(N) $ such that $ c_k(N)\to 1 $ as $ N\to \infty $. Let $ (\vartheta_n(\alpha))_{n\ge 1}\;(\alpha \in J) $ be a parametrised family of sequences such that the map $ \alpha \mapsto \vartheta_n(\alpha) $ is continuous for each fixed $ n\ge1 $. Assume that there exists $ \rho>0 $ such that for all $ f\in C_c^\infty (\mathbb{R}^{k-1}) $\[ \int_J\left(R_{k}\left(f,\alpha,N\right)-c_{k}\left(N\right)\int_{\mathbb{R}^{k-1}}f\left({\bf x}\right)\,\text{d}{\bf x}\right)^{2}\,\mathrm{d}\alpha = O(N^{-\rho}) \]as $ N\to \infty $. Then for almost all $ \alpha \in J $, the sequence $(\{\vartheta_n(\alpha)\})_{n\ge 1}$ has Poissonian $ k$-point correlation.
\end{prop}

\begin{proof}
Poissonian $ k $-point correlation is first established along a polynomially sparse subsequence $ N_m $ using the Borel-Cantelli lemma (or an analogous argument). This is extended to Poissonian $ k $-point correlation along the full sequence by a simple sandwiching argument, using the fact that $ \lim\limits_{m\to\infty}N_{m+1}/N_m =1 $. For the full details see \cite{YT}.
\end{proof}

\begin{proof}[Proof of Theorem \ref{thm: main}]
Theorem \ref{thm: main} follows upon letting $ \vartheta_n(\alpha)=e^{\alpha a_{n}} $, $ J=\J $, $ c_k(N)=C_k(N) $ (recall \eqref{eq:C_k(N)}) and $ \rho=-1+\varepsilon $ in Proposition \ref{prop:L2->a.s.}.
\end{proof}


\begin{thebibliography}{10}

\bibitem{A} C. Aistleitner: \emph{Quantitative uniform distribution results for geometric progressions},  Israel J. Math. 204(1): 155--197, 2014.

\bibitem{AB}
C. Aistleitner and S. Baker: \emph{On the pair correlations of powers of real numbers,} Israel J. Math., to appear. \href{https://arxiv.org/abs/1910.01437}{arXiv:1910.01437}

\bibitem{ALP} C. Aistleitner, T. Lachmann and F. Pausinger: \emph{Pair correlations and equidistribution},  J. Number Theory 182: 206--220, 2018.

\bibitem{Ba} S. Baker: \emph{Equidistribution results for self-similar measures}, arXiv:2002.11607.
\bibitem{B} Y. Bugeaud: \emph{Distribution modulo one and Diophantine approximation}, Cambridge Tracts in Mathematics, 193. Cambridge University Press, Cambridge, 2012.
\bibitem{D1} A. Dubickas: \emph{On the powers of 3/2 and other rational numbers}, Math. Nachr. 281(7): 951--958, 2008.

\bibitem{D2} A. Dubickas: \emph{Powers of a rational number modulo 1 cannot lie in a small interval}, Acta Arith. 137(3): 233--239, 2009.

\bibitem{EMV} D. El-Baz, J. Marklof and I. Vinogradov. \emph{The two-point correlation function of the fractional parts of $\sqrt{n}$ is Poisson},  Proc. Amer. Math. Soc. 143(7): 2815--2828, 2015.

\bibitem{EM}N. D. Elkies and C. T. McMullen: \emph{Gaps in $\sqrt{n}$ mod $1$ and ergodic theory}, Duke Math. J. 123(1):
95--139, 2004.

\bibitem{FLP} L. Flatto, J. Lagarias and A. Pollington: \emph{On the range of fractional parts $\{\xi (p/q)^n \}$}, Acta Arith. 70(2): 125--147, 1995.

\bibitem{GL} S. Grepstad and G. Larcher: \emph{On pair correlation and discrepancy},  Arch. Math. (Basel) 109(2): 143--149, 2017.

\bibitem{HB}D. R. Heath-Brown: \emph{Pair correlation for fractional parts of $\alpha n^{2}$},
Math. Proc. Cambridge Philos. Soc., 148(3):385--407, 2010.

\bibitem{Knuth}D. E. Knuth: \emph{The Art of Computer Programming:
Volume 1: Fundamental Algorithms} (3rd ed.), Addison Wesley, 1997.

\bibitem{Koksma}J. F. Koksma: \emph{Ein mengentheoretischer Satz über die Gleichverteilung modulo Eins},
Compositio Math., 2: 250--258, 1935.


\bibitem{KR} P. Kurlberg and Z. Rudnick: \emph{The distribution of spacings between quadratic residues.} Duke Math. J. 100(2): 211--242, 1999.

\bibitem{mark} J. Marklof: \emph{Pair correlation and equidistribution on manifolds},  Monatsh. Math. 191(2): 279--294, 2020.

\bibitem{NT} H. Niederreiter and R. Tichy: \emph{Solution of a problem of Knuth on complete uniform distribution of sequences},  Mathematika 32(1): 26--32, 1985.

\bibitem{RS} Z. Rudnick and P. Sarnak: \emph{The pair correlation function of fractional
parts of polynomials}, Comm. Math. Phys., 194(1): 61--70, 1998.

\bibitem{RSZ} Z. Rudnick, P. Sarnak and A. Zaharescu: \emph{The distribution of spacings
between the fractional parts of $n^{2}\alpha$}, Invent. Math., 145(1):37--57,
2001.

\bibitem{RZ} Z. Rudnick and A. Zaharescu: \emph{A metric result on the pair correlation of fractional parts of sequences}. Acta Arith. 89(3): 283--293, 1999.

\bibitem{RZ2} Z. Rudnick and A. Zaharescu: \emph{The distribution of spacings between fractional parts of lacunary sequences}. Forum Math. 14(5): 691--712, 2002.

\bibitem{TW} N. Technau and W. Walker: \emph{On the triple correlations of fractional parts of $n^2 \alpha$}, \href{https://arxiv.org/abs/2005.01490}{arXiv:2005.01490}

\bibitem{YT} N. Technau and N. Yesha: \emph{On the correlations of
$n^{\alpha}$ mod $1$}, \href{https://arxiv.org/pdf/2006.16629.pdf}{arXiv:2006.16629}.

\bibitem{TD} J. L. Truelsen: \emph{Divisor problems and the pair correlation for the
	fractional parts of $n^{2}\alpha$}, Int. Math. Res. Not., (16):3144--3183,
2010.

\bibitem{Weyl}H.
Weyl\emph{: \"{U}ber die Gleichverteilung von Zahlen mod. Eins}, Math.
Ann., 77(3): 313--352, 1916.
\end{thebibliography}
\end{document}